\documentclass[12pt]{amsart}
\usepackage{verbatim, amsmath,amsthm,amssymb, amscd, amsfonts}
\newtheorem{theorem}{Theorem}[section]
\newtheorem{lemma}[theorem]{Lemma}
\newtheorem{proposition}[theorem]{Proposition}
\theoremstyle{definition}
\newtheorem{definition}[theorem]{Definition}
\newtheorem{example}[theorem]{Example}
\newtheorem{rema}[theorem]{Remark}
\newtheorem{coro}[theorem]{Corollary}

\newtheorem{que}[theorem]{Question}

\begin{document}
\newcommand{\nc}{\newcommand}
\nc{\rnc}{\renewcommand} \nc{\nt}{\newtheorem}


\nc{\TitleAuthor}[2]{\nc{\Tt}{#1}%
    \nc{\At}{#2}%
    \maketitle%
}


\nc{\Hom}{\operatorname{Hom}} \nc{\Mor}{\operatorname{Mor}} \nc{\Aut}{\operatorname{Aut}}
\nc{\Ann}{\operatorname{Ann}} \nc{\Ker}{\operatorname{Ker}} \nc{\Trace}{\operatorname{Trace}}
\nc{\Char}{\operatorname{Char}} \nc{\Mod}{\operatorname{Mod}} \nc{\End}{\operatorname{End}}
\nc{\Spec}{\operatorname{Spec}} \nc{\Span}{\operatorname{Span}} \nc{\sgn}{\operatorname{sgn}}
\nc{\Id}{\operatorname{Id}} \nc{\Com}{\operatorname{Com}}

\nc{\nd}{\mbox{$\not|$}} 
\nc{\nci}{\mbox{$\not\subseteq$}}
\nc{\scontainin}{\mbox{$\mbox{}\subseteq\hspace{-1.5ex}\raisebox{-.5ex}{$_\prime$}\hspace*{1.5ex}$}}


\nc{\R}{{\sf R\hspace*{-0.9ex}\rule{0.15ex}%
    {1.5ex}\hspace*{0.9ex}}}
\nc{\N}{{\sf N\hspace*{-1.0ex}\rule{0.15ex}%
    {1.3ex}\hspace*{1.0ex}}}
\nc{\Q}{{\sf Q\hspace*{-1.1ex}\rule{0.15ex}%
       {1.5ex}\hspace*{1.1ex}}}
\nc{\C}{{\sf C\hspace*{-0.9ex}\rule{0.15ex}%
    {1.3ex}\hspace*{0.9ex}}}
\nc{\Z}{\mbox{${\sf Z}\!\!{\sf Z}$}}


\newcommand{\gd}{\delta}
\newcommand{\sub}{\subset}
\newcommand{\cntd}{\subseteq}
\newcommand{\go}{\omega}
\newcommand{\Pa}{P_{a^\nu,1}(U)}
\newcommand{\fx}{f(x)}
\newcommand{\fy}{f(y)}
\newcommand{\gD}{\Delta}
\newcommand{\gl}{\lambda}
\newcommand{\half}{\frac{1}{2}}
\newcommand{\ga}{\alpha}
\newcommand{\gb}{\beta}
\newcommand{\gga}{\gamma}
\newcommand{\ul}{\underline}
\newcommand{\ol}{\overline}
\newcommand{\Lrraro}{\Longrightarrow}
\newcommand{\equi}{\Longleftrightarrow}
\newcommand{\gt}{\theta}
\newcommand{\op}{\oplus}
\newcommand{\Op}{\bigoplus}
\newcommand{\CR}{{\cal R}}
\newcommand{\tr}{\bigtriangleup}
\newcommand{\grr}{\omega_1}
\newcommand{\ben}{\begin{enumerate}}
\newcommand{\een}{\end{enumerate}}
\newcommand{\ndiv}{\not\mid}
\newcommand{\bab}{\bowtie}
\newcommand{\hal}{\leftharpoonup}
\newcommand{\har}{\rightharpoonup}
\newcommand{\ot}{\otimes}
\newcommand{\OT}{\bigotimes}
\newcommand{\bwe}{\bigwedge}
\newcommand{\gep}{\varepsilon}
\newcommand{\gs}{\sigma}
\newcommand{\OO}{_{(1)}}
\newcommand{\TT}{_{(2)}}
\newcommand{\FF}{_{(3)}}
\newcommand{\minus}{^{-1}}
\newcommand{\CV}{\cal V}
\newcommand{\CVs}{\cal{V}_s}
\newcommand{\slp}{U_q(sl_2)'}
\newcommand{\olp}{O_q(SL_2)'}
\newcommand{\slq}{U_q(sl_n)}
\newcommand{\olq}{O_q(SL_n)}
\newcommand{\un}{U_q(sl_n)'}
\newcommand{\on}{O_q(SL_n)'}
\newcommand{\ct}{\centerline}
\newcommand{\bs}{\bigskip}
\newcommand{\qua}{\rm quasitriangular}
\newcommand{\ms}{\medskip}
\newcommand{\noin}{\noindent}
\newcommand{\raro}{\rightarrow}
\newcommand{\alg}{{\rm Alg}}
\newcommand{\rcom}{{\cal M}^H}
\newcommand{\lcom}{\,^H{\cal M}}
\newcommand{\rmod}{\,_R{\cal M}}
\newcommand{\qtilde}{{\tilde Q^n_{\gs}}}
\nc{\e}{\overline{E}} \nc{\K}{\overline{K}} \nc{\gL}{\Lambda}
\newcommand{\tie}{\bowtie}
\newcommand {\h}{\widehat}
\newcommand {\tl}{\tilde}
\newcommand{\tri}{\triangleright}
\nc{\ad}{_{\dot{ad}}}
\nc{\coad}{_{\dot{{\rm coad}}}}
\nc{\ov}{\overline}

\title{Are we counting or measuring something?}
\author{Miriam Cohen}
\address {Department of Mathematics\\
Ben Gurion University, Beer Sheva, Israel}
\email {mia@math.bgu.ac.il}
\author{Sara Westreich}
\address{Department of Management\\
Bar-Ilan University,  Ramat-Gan, Israel}
\email{swestric@biu.ac.il}
\thanks {This research was supported by the ISRAEL SCIENCE FOUNDATION, 170-12.}


\date{Apr-2013}
\begin{abstract}
Let $H$ be a semisimple Hopf algebras over an algebraically closed field $k$ of characteristic $0.$ We define Hopf algebraic analogues of  commutators and their generalizations and show how they are related to $H',$ the Hopf algebraic analogue of the commutator subgroup.  We introduce a family of central elements of $H',$ which on one hand generate $H'$ and on the other hand give rise to a family of functionals on $H.$ When $H=kG,\,G$ a finite group, these functionals are counting functions on $G.$ It is not clear yet to what extent they measure any specific invariant of the Hopf algebra. However, when $H$ is quasitriangular they are at least characters on $H.$
\end{abstract}
\maketitle
\section*{Introduction}
Commutators and commutator subgroups  are some of the most fundamental concepts in group theory. These subgroups measure how far the group is from being commutative. It was Frobenius who proved early on that a function on a finite group $G,$ that counts the number of ways an element of $G$ can be realized as a commutator, is a character of $G.$ This was done by giving an explicit formula for this counting function. Generalizations of this formula were proved throughout the years (e.g. \cite{ta,km,av}). The approaches varied from a direct approach through the use of a symmetric bilinear form and its associated Casimir element, to  the use of distribution functions which are uniform on conjugacy classes of $G.$

In this paper we define Hopf algebraic analogues of  commutators and their generalizations and show how they are related to $H',$ the Hopf algebraic analogue of the commutator subgroup.  We introduce a family of elements in $H'$ denoted by $z_n,\,n>1,$ which arise from the idempotent integral of $H.$ This family consists of powers of the $S$-fixed central invertible element $z_2,$
$$z_2=\sum \gL_1\gL'_1S\gL_2S\gL'_2$$
where $\gL,\gL'$ are two copies of the idempotent integral of $H.$

The elements $\{z_n\}$
  are shown  to give rise to Hopf algebraic analogues of the various counting functions for groups. On the other hand they are shown to be central Casimir elements associated to certain symmetric bilinear forms and Higman maps on $H$ or on its center.

A different characterizations of $z_n$ is given when $H$ is also assumed to be almost cocommutative. In this situation the  $z_n$'s are related to so called {\it generalized class sums} for $H.$ Information about functionals related to iterated commutators can be deduced from the character table of $H.$

\bigskip We use the following notations. Let $H$ be a $d$-dimensional semisimple Hopf algebras over an algebraically closed  field $k$ of characteristic $0$ with an idempotent integral $\gL.$ Denote by $\{E_i\}_{0\le i\le n-1}$ the set of primitive central idempotents of $H.$
Let $d_i$ denote the degree of the irreducible character $\chi_i$ and let $R(H)={\rm Sp}_k\{\chi_i\}.$ Then $\chi_H=\gl=\sum d_i\chi_i$ is an integral for $H^*.$

Let $\Psi :H_{H^*}\rightarrow H^*_{H^*}$ be the Frobenius map given by:
$$\Psi(h)=\gl\leftharpoonup S(h).$$

The commutator subalgebra $H'$  is a normal left coideal subalgebra of  $H$ for which  $H/(H{H'}^+)$ is commutative and it is minimal with respect to this property.

\medskip In Section $2$ we define the commutator $\{a,b\}$ for $a,b\in H$   as follows:
$$\{a,b\}=\sum a_1b_1Sa_2Sb_2.$$
We define also the general commutator $\{a^1\dots ,a^n\},\;a^i\in H$ as follows:
$$ \{a^1,\dots, a^n\}=\sum a^1_1\cdots a^n_1Sa^1_2\cdots Sa^n_2.$$
We show that the general commutator can always be obtained as a sum of products of commutators.

Of special interest are commutators related to the idempotent integral $\gL$ of $H.$
Let $\gL^i$ be a copy of the idempotent $\gL.$ Set
$$z_n=\sum\gL^1_1\cdots\gL^n_1S\gL^1_2\cdots S\gL^n_2\qquad z_0=1.$$

We show:

\medskip\noin{\bf Theorem \ref{central22}}:Let $H$ be a semisimple Hopf algebra over an algebraically closed field $k$ of characteristic $0.$ Then for all $k,n\ge 0,$
$$z_{2k+1}=z_{2k},\qquad z_n=z_2^{\frac{n-n({\rm mod}2)}{2}}\in Z(H).$$

\medskip The commutator $z_2$ has a very nice form:

\medskip\noin{\bf Theorem \ref{z2}}:
Let $H$ be a semisimple Hopf algebra over an algebraically closed field $k$ of characteristic $0.$  Then  $$z_2=\sum \gL^1_1\gL_1^2S\gL^1_2S\gL_2^2=\sum_i \frac{1}{d_i^2} E_i\in Z(H).$$

\medskip The elements $z_n$ serve as   indicators for the commutativity of $H$ as follows:

\medskip\noin{\bf Theorem \ref{z2com}}:
Let $H$ be a semisimple Hopf algebra over an algebraically closed field $k$ of characteristic $0.$ Then $H$ is commutative if and only if $z_n\in k$ for some $n\ge 2,$ if and only if $z_n\in k$ for all $n\ge 2.$

\medskip In fact $z_n$ generates $H'$ in the following sense:

\medskip\noin{\bf Theorem \ref{z2com2}}:
Let $H$ be a semisimple Hopf algebra over an algebraically closed field $k$ of characteristic $0$ and let $n\ge 2.$ Then the commutator sublagebra of $H$ is the algebra generated by the left coideal
$$H'_n=z_n\leftharpoonup H^*.$$

\medskip In Section $3$ we discuss ``counting functions'' for $H.$  An important set of functionals on finite groups are the so called counting functions.  As in \cite{av}, let $w$ be a word in the free groups on $r$ letters. Substituting $(g_1,\cdots,g_r)\rightarrow w(g_1,\dots,g_r)$ defines a natural function from $G^r$ to $G.$ Let $N_{w}(g)$ denote the number of solutions to $w(g_1,\dots,g_r)=g.$ The function $N_w$ is a class function, called a counting function.

We discuss analogues of the four following counting functions: (i)Root functions. (ii) Frobenius counting function for commutators. (iii) Counting functions for generalized commutators. (iv) Counting functions for iterated commutators.

The Frobenius counting function for commutators,
$$f_{rob}=\sum_i \frac{d}{d_i}\chi$$
is realized in our setup as:
$$f_{rob}=d\Psi(z_2).$$

More generally, the convolution power of $f_{rob}$ is the counting function for the generalized commutators, and for products of commutators. The convolution product $\bullet$ inside $H^*$ is defined for $p,q\in H^*$ as follows:
$$p\bullet p'=d\Psi(\Psi\minus(p)\Psi\minus(q)).$$

We show:

\medskip\noin{\bf Theorem \ref{bulletl}:}
Let $H$ be a semisimple Hopf algebra over an algebraically closed field $k$ of characteristic $0.$  Then for any $l>0,$
$$d^{2l-1}\Psi(z_{2l})=\sum_i\frac{d^{2l-1}}{d_i^{2l-1}}\chi_i=\underbrace{f_{rob}\bullet f_{rob}\cdots\bullet f_{rob}}_{l}.$$
The function in the middle is called $f_n$ and for groups it equals the function that counts general commutators of length $2n$ or $2n+1.$ It also equals the function that counts products of $n$ commutators.

\medskip
The commutators $z_n$ are realized also as Casimir elements for specific symmetric forms on $H$ and on $Z(H).$

\medskip The analogue of the counting function for iterated commutators of a group  is defined as:
$$f_{\{\{x,y\},z\}}=d^2\sum_{i,j} \frac{1}{d_id_j}\langle \chi_i s(\chi_i)\chi_j,\gL\rangle\chi_i.$$

We show:

\medskip\noin{\bf Theorem \ref{z2lamda}}:
Let $H$ be a semisimple hopf algebra over an algebraically closed fiels of characteristic $0.$  Then:
$$f_{\{\{x,y\},z\}}=d^2\sum_{i,j} \frac{1}{d_i}\langle \chi_i s(\chi_i),z_2\rangle\chi_i=d^2\Psi(\{z_2,\gL\}).$$

\medskip

In Section $4$ we focus on almost cocommutative Hopf algebras. Let $\{\eta_j\}$ be the set of normalized class sums of $H.$

Observe that for a group $G,$ we have $g$ is a commutator if and only if $f_{rob}(g)\ne 0.$ Moreover, $g\in G'$ if and only if $f_n(g)\ne 0$ for some $n\ge 1.$ Similarly, $g$ is an iterated commutator if and only if $f_{\{\{x,y\}z\}}(g)\ne 0.$   Hence also for a class sum $C,\;f_{rob}(C)\ne 0$ (respectively $f_n(C)\ne 0,\, f_{\{\{x,y\}z\}}(C)\ne 0$) if and only if $C$ is a sum of commutators ( respectively $C\in kG',\;C$ is a sum of iterated commutators).

\medskip In this spirit we show:

\medskip\noin {\bf Theorem \ref{eta}}:
Let $H$ be an almost cocommutative semisimple Hopf algebra over an algebraically closed field of characteristic $0,$ let $\eta_j$ be a normalized class sum and let $f_n$ be as in \eqref{fng}. Then
$$\eta_j\in z_2^n\leftharpoonup H^*\Longleftrightarrow \langle f_n,\eta_j\rangle\ne 0.$$

\medskip For Hopf algebras that satisfy the condition in Kaplansky's conjecture,  that is, if $d_i|d$ for all $i,$ we have:

\medskip\noin{\bf Theorem \ref{kap}}:
Let $H$ be a semisimple Hopf algebra over an algebraically closed field of characteristic $0,$ and assume $d_i|d$ for any irreducible $H$-module $V_i.$ Then the  functions $f_{rob},\,f_n$ and $f_{\{\{x,y\},z\}}$  are characters.

\section{Preliminaries} Throughout this paper $H$ is a $d$-dimensional Hopf algebra over an algebraically closed field $k$ of characteristic $0.$  We denote by $S$ and $s$ the antipodes of $H$ and $H^*$ respectively and $\gL$ and $\gl$ the left and right integrals of $H$ and $H^*$ respectively so that $\left\langle \gl,\gL \right\rangle  = 1.$
Denote by
 $Z(H)$ the center of $H.$

For any  algebra $A,\;A^*=\hom(A,k)$  becomes a right and left $A$-module by the {\it hit} actions $\leftharpoonup$ and $\rightharpoonup$ defined for all $a\in A,\,p\in A^*,$
$$\left\langle p\leftharpoonup a,a'\right\rangle =\left\langle p,aa'\right\rangle \qquad \left\langle a\rightharpoonup p,a'\right\rangle =\left\langle p,a'a\right\rangle $$

In particular if $A=H$ is a finite dimensional Hopf algebras then $H^*$ is an algebra and thus $H$ becomes a left and right  $H^*$ module.

Denote by $_{\dot{ad}} $ the left adjoint action of $H$ on itself, that is, for all $a,h\in H,$
$$h_{\dot{ad}}  a=\sum h_1aS(h_2).$$
Then $$\gL\ad H\subset Z(H)$$ and if $H$ is semisimple then equality holds.

A subalgebra $A$ of $H$ is called {\bf normal} if it is stable under the left adjoint action.

Let $D(H)$ denote the Drinfeld double of the Hopf algebra $H.$ It is not hard to see that $H$ is a $D(H)$-module with respect to the adjoint action of $H$ on itself and the right {\rm hit} action of $H^*$ on $H.$

\medskip
Denote by $R(H)$ the $k$-span of all irreducible characters of $H.$  It is an algebra (called the character algebra) which is contained in the algebra of all cocommutative elements of $H^*$ and equal to it when $H$ is semisimple.

Recall that $H$ is a Frobenius algebra. One defines a Frobenius map $\Psi
:H_{H^*}\rightarrow H^*_{H^*}$ by
\begin{equation}\label{psi}\Psi(h)=\gl\leftharpoonup S(h)\end{equation}
where $H^*$ is a right $H^*$-module under multiplication and $H$ is a right $H^*$-module under right {\it hit}. If $H$ is semisimple then
$$\Psi(Z(H))=R(H).$$

For a finite-dimensional Hopf algebra $H$ we have for all $p\in H^*,$
$$\Psi\minus(p)=\gL\leftharpoonup p,$$

Let $H$ be a semisimple Hopf algebra over $k.$ Let  $\{k=V_0,\dots V_{n-1}\}$  be a complete set of non-isomorphic irreducible $H$-modules, ${\rm Irr}(H)=\{\gep=\chi_0,\dots\chi_{n-1}\}$ the corresponding characters  and  $\{E_0,\dots E_{n-1}\}$  the associated central primitive idempotents of $H$ where $E_0=\gL$ is the idempotent integral of $H.$   Let $\dim V_i=d_i,$ then
$$\gl=\chi_H=\sum_{i=0}^{n-1}d_i\chi_i,$$
where $\chi_H$ denotes the character of the left regular representation.

A well known result of Larson \cite{la} is the orthogonality of characters, that is,
\begin{equation}\label{lacha}
\langle\chi_i s(\chi_j),\gL\rangle =\gd_{ij}.\end{equation}

The orthogonality  of characters implies in particular (see also \cite[Cor.4.6]{sc}):
 \begin{equation}\label{dual1}\left\langle \chi_i,E_j\right\rangle =\gd_{ij}d_j,\quad \Psi(E_i)=d_is(\chi_i),\quad
\chi_i\leftharpoonup E_j=\gd_{ij}\chi_i \end{equation}
for all  $0\le i,j\le n-1.$
In particular, $\{\chi_i\},\,\{\frac{1}{d_j}E_j\}$ are dual bases of $R(H)$  and $Z(H)$ respectively. Hence we have for each $z\in Z(H),\,p\in R(H)$
\begin{equation}\label{z}z=\sum_i \frac{1}{d_i}\left\langle \chi_i,z\right\rangle E_i,\qquad p=\sum_i\frac{1}{d_i}\langle p,E_i\rangle\chi_i.\end{equation}
By \eqref{dual1} we have
\begin{equation}\label{act}
\chi_i\leftharpoonup z = \frac{1}{d_i}\left\langle \chi_i,z\right\rangle\chi_i\end{equation}
for all $i.$

\medskip A finite-dimensional algebra $A$ over $k$ is a {\bf symmetric algebra} if there exists a non-degenerate associative symmetric bilinear form $\gb:A\ot A\rightarrow k.$

It is well known there exists a bijective correspondence between symmetric bilinear forms $\gb$ on $A$ and elements $t_{\gb}\in A^*$ so that  $\langle t_{\gb},ab\rangle=\langle t_{\gb},ba\rangle$ for all $a,b\in A$ and $t_{\gb}\leftharpoonup A=A^*,$ (that is $t_{\gb}$ is a generator of $A^*$ as an $A$-module).    The correspondence is given as follows: Given $t,$ a genrator of $A^*,$ we define $\gb_t(a,b)=\langle t,ab\rangle$ and conversely, given $\gb$ the the associated generator is $t_{\gb}\in A^*$ by $\langle t_{\gb},a\rangle=\gb(a,1).$

For each $\gb$ there exists $\{r_i\},\{l_i\}$ which forms a dual basis for $\gb,$ that is $\gb(r_i,l_j)=\gd_{ij},$ for $1\le i,j\le \dim A.$

The  {\bf Casimir element} with respect to $\gb$  is
$$\sum_ir_i\ot l_i=\sum l_i\ot r_i.$$
It does not depend on the choice of the dual basis for $\gb.$ Moreover, for all $a\in A,$
\begin{equation}\label{ccasimir}
\sum_ir_ia\ot l_i=\sum r_i\ot al_i\quad{and}\quad\sum ar_i\ot l_i=\sum r_i\ot l_ia.
\end{equation}

The  {\bf central Casimir} element with respect to $\gb$ denoted by $Cas_{\gb}$ is
$$Cas_{\gb}=\sum_ir_il_i\in Z(A).$$

Any other symmetric form $\gb'$ with corresponding $t',\,r_i',\,l_i'$ satisfies $t'=t\leftharpoonup u$ where $u$ is an invertible element in the center of $A$ and
\begin{equation}\label{connection} r_i'=r_iu\minus,\;l_i'=l_i,\;Cas_{\gb'}=\sum r_iu\minus l_i.\end{equation}

Any symmetric form $\gb$ defines an Higman map $\tau_{\gb}:H\rightarrow Z(H)$ by:
\begin{equation}\label{tau}\tau(h)=\sum r_ihl_i.\end{equation}

Finite-dimensional semisimple algebras are always symmetric. Also, if $H$ be a semisimple Hopf algebra over an algebraically closed field $k$ then it was shown in\cite{os} that  $H$ has a natural symmetric bilinear form $\gb$ given by:
\begin{equation}\label{bilinear}\gb(h,h')=\left\langle\gl,hh'\right\rangle,\end{equation}
for all $h,h'\in H.$  The  corresponding Casimir element is

\begin{equation}\label{casimir}Cas_{\gb}=\sum\gL_1\ot S\gL_2.\end{equation}

It follows from \eqref{ccasimir} (it can also be implied directly from  \cite[Lemma 1.2]{lr}) that for all $a\in H,$
\begin{eqnarray}\label{lara}
(i)&&\sum \gL_1\ot aS\gL_2=\sum \gL_1a\ot S\gL_2 \\
\nonumber (ii)&&\sum a\gL_1\ot S\gL_2=\sum \gL_1\ot S\gL_2a.\end{eqnarray}

\section{Commutators for Hopf algebras}

The commutator subalgebra $H'$ of a semisimple Hopf algebra $H$ was first defined in \cite{bu1}. It is a normal left coideal subalgebra of  $H$ for which  $H/(H{H'}^+)$ is commutative and it is minimal with respect to this property. For $H=kG$ one has $H'=kG'.$

Based on \cite[Prop.1.14]{cw5}, it  is not hard to see that
\begin{equation}\label{comma}H'= \{h\in H\,|\,\gs\rightharpoonup h = h\;\forall \gs\in G(H^*)\}.\end{equation}

Generalizing from groups we describe $H'$ in terms of Hopf algebraic commutators. Let $H$ be any Hopf algebra over $k.$ For $a,b\in H,$ define their commutator $\{a,b\}$ as:
\begin{equation}\label{comh}\{a,b\}=\sum a_1b_1Sa_2Sb_2.\end{equation}
It is straightforward to check that
\begin{equation}\label{com1}ab=\sum\{a_1,b_1\}b_2a_2\end{equation} for all $a,b\in H.$
Define
\begin{equation}\label{com}\it{Com}={\rm span}_k\{\{a,b\}\,|\,a,b\in H\}.\end{equation}

\begin{lemma}\label{coideal} Let $H$ be a Hopf algebra over $k,$ then  $\it{Com}$ is a left coideal of $H.$
\end{lemma}
\begin{proof}
For each commutator $\{a,b\},$
\begin{equation}\label{left}\gD(\{a,b\})=\sum a_1b_1Sa_3Sb_3\ot \{a_2,b_2\}\in H\ot \it{Com}.\end{equation}
\end{proof}

As a result we show:

\begin{proposition}\label{h'} Let $H$ be a semisimple Hopf algebra then the commutator subalgebra $H'$ of $H$ is the algebra generated by $\it{Com}.$ Thus $H$ is commutative if and only if $\it{Com}=k.$
\end{proposition}
\begin{proof}

Let $N$ denote the algebra generated by $\it{Com}.$ Since $\it{Com}$ is a left coideal, it follows that $N$ is a left coideal subalgebra.

Next we show that $N$ is normal. Since $H$ is an $H$-module algebra under the left adjoint action it is enough to check it on {\it Com}. Indeed, by \eqref{com1} we have for all $h,x,y\in H,\,$
$$\sum \underbrace {h_1\{x,y\}}_{a}\underbrace{Sh_2}_{b}=\sum \{h_1\{x,y\}_1,Sh_4\}Sh_3h_2\{x,y\}_2=\sum \{h_1\{x,y\}_1,Sh_2\}\{x,y\}_2.$$
The first term is a commutator by definition, while the second one belongs to $\it{Com}$ by Lemma \ref{coideal}.

Now, for any  $\gs\in G(H^*),$   we have:
$$\gs\rightharpoonup\sum a_1b_1Sa_2Sb_2=\sum\left\langle \gs,a_2b_2Sa_3Sb_3\right\rangle a_1b_1Sa_4Sb_4=\sum a_1b_1Sa_2Sb_2$$
The last equality follows from the fact that $\gs$ is multiplicative on $H.$ By the definition of $H'$ we have that $\it{Com}\subset H'$ and since $H'$ is an algebra, we have also $N\subset H'.$

Let $\ol{H}=H/HN^+.$ Then by \eqref{com1} we have  for all $\ol{a},\ol{b}\in  \ol{H},$
$$\ol{ab}=\sum\{\ol{a_1,b_1}\}\ol{b_2a_2}.$$ But $\{\ol{a_1,b_1}\}=\langle\gep, a_1b_1\rangle,$ hence
$\ol{ab}=\ol{ba}.$
\end{proof}

Of special interest is the following commutator:
\begin{lemma}\label{zlamda}
Let  $z\in Z(H).$ Then $\{z,\gL\}\in Z(H).$
\end{lemma}
\begin{proof}
For all $h\in H$ we have:
\begin{eqnarray*}
\lefteqn{h\sum z_1\gL_1Sz_2S\gL_2=}\\
&=& \sum h_1 z_1\gL_1S(Sh_3h_2z_2)S\gL_2\\
&=& \sum z_1h_1 \gL_1S(Sh_3z_2h_2)S\gL_2\qquad(\text{since }\,z\in Z(H))\\
&=& \sum z_1h_1 \gL_1Sh_2Sz_2h_3S\gL_2\\
&=&\sum z_1h \gL_1Sz_2S\gL_2\qquad\text{(by \eqref{lara}(i) with $a=h_3$)}\\
&=&\sum z_1\gL_1Sz_2S\gL_2h\qquad(\text{since }\,\gL\cdot Sz_2\in Z(H))
\end{eqnarray*}
\end{proof}

\medskip
We consider now generalized commutators.
\begin{definition} For Hopf algebras $H$ and $n>1,$ the {\bf $n$-th commutator} $\{a^1,\dots, a^n\}$ where $a^1,\dots a^n\in H,$ is given by:
$$ \{a^1,\dots, a^n\}=\sum a^1_1\cdots a^n_1Sa^1_2\cdots Sa^n_2.$$
Define also the subspace generated by $n$-th commutators:
$$\it{Com}_n={\rm span}_k\{\{a^1,\dots, a^n\}\,|\, a^i\in H\;\text{for all}\; 1\le i\le n\}.$$
\end{definition}
Note $\it{Com}_2$ coincides with the subspace of commutators,  $\it{Com}.$ Moreover, $\it{Com}_n\subseteq \it{Com}_m$ for all $2\le n\le m.$ Hence proposition \ref{h'} can be generalized and we obtain:
\begin{proposition}\label{hn}
Let $H$ be a semisimple Hopf algebra. Then the commutator subalgebra $H'$ of $H$ is the algebra generated by $\it{Com}_n.$ In particular, every generalized commutator is a linear combination of products of commutators.
\end{proposition}
\begin{proof}
Since $\it{Com}_2\subseteq \it{Com}_n$ it follows that the algebra generated by $\it{Com}_2$ is contained in the algebra generated by $\it{Com}_n.$ By Proposition \ref{h'} it follows that $H'$ is contained in the algebra generated by $\it{Com}_n.$ Conversely, a direct computation shows that for any $\gs\in H^*,$
$$\gs\rightharpoonup \{a^1,\dots,a^n\} = \{a^1,\dots,a^n\}.$$ Hence equality holds.

The last part is straightforward.
\end{proof}

\medskip
Of special interest will be commutators and generalized commutators related to the idempotent integral $\gL$ of $H.$
Let $\gL^i$ be a copy of the idempotent $\gL.$ Set
\begin{equation}\label{zn}z_n=\sum\gL^1_1\cdots\gL^n_1S\gL^1_2\cdots S\gL^n_2\qquad z_0=1.\end{equation}
Define a map $Z_n:H\longrightarrow H$ by:
\begin{equation}\label{znh}Z_n(h)= \sum\gL^1_1\cdots\gL^n_1hS\gL^1_2\cdots S\gL^n_2\qquad Z_0(h)=h.\end{equation}

We show now:
\begin{proposition}\label{central2}
Let $H$ be a semisimple Hopf algebra.   Then:

\medskip\noin {\rm 1}. For all $h\in H,\,n>1$
$$Z_n(h)=z_2Z_{n-2}(h).$$

\medskip\noin {\rm 2}. $Z_{2k+1}(h)\in Z(H)$ and $Z_{2k}(h)=z_{2k}h$ for all $k\ge 0,\,h\in H.$
\end{proposition}
\begin{proof}

\noin 1.
\begin{eqnarray*}\lefteqn{Z_n(h)=}\\
&=&\sum \gL^1_1\gL^2_1\underbrace{\gL^3_1\cdots\gL^{n}_1h}_{a}S\gL^1_2S\gL^2_2\cdots S\gL^{n}_2\\
&=& \sum \gL^1_1\underbrace{\gL^3_1\cdots\gL^{n}_1h}_a\gL^2_1S\gL^1_2S\gL^2_2\cdots S\gL^{2n}_2\qquad\text{(by \eqref{lara}(i) with $\gL=\gL^1$)}\\
&=& \sum \gL^1_1\gL^2_1S\gL^1_2S\gL^2_2\underbrace{\gL^3_1\cdots\gL^{n}_1h}_{a}S\gL^3_2\cdots S\gL^{2n}_2\qquad\text{(by \eqref{lara}(ii) with $\gL=\gL^2$)}\\
&=&z_2 Z_{n-2}(h)
\end{eqnarray*}

\noin 2. To see the odd case we first show that $z_2\in Z(H).$ Indeed, by \eqref{lara}(ii) with $\gL=\gL^1$ and then \eqref{lara}(i) with $\gL=\gL^2$, we have for all $h\in H,$
$$hz_2=\sum h\gL^1_1\gL^2_1S\gL^1_2S\gL^2_2=\sum \gL^1_1\gL^2_1S\gL^1_2hS\gL^2_2=\sum \gL^1_1\gL^2_1hS\gL^1_2S\gL^2_2=Z_2(h)=z_2h.$$
The last equality follows from part 1.

Since $Z_1(h)=\gL\ad h\in Z(H),$ the result follows  from part 1 by induction on $k.$

The  even case follows by induction since $Z_0(h)=h.$
\end{proof}
As a consequence we obtain the following theorem:
\begin{theorem}\label{central22}
Let $H$ be a semisimple Hopf algebra over an algebraically closed field $k$ of characteristic $0.$ Then for all $k,n\ge 0,$
$$z_{2k+1}=z_{2k}\qquad z_n=z_2^{\frac{n-n({\rm mod}2)}{2}}\in Z(H).$$
\end{theorem}
\begin{proof}
Taking $h=1$ in Proposition \ref{central2} yields $$z_{n}=Z_n(1)=z_2Z_{n-2}(1)=z_2z_{n-2}.$$
By induction $z_{2k}=z_2^k.$ Since $z_1=\sum\gL_1S\gL_2=1$ it follows that $z_3=z_2z_1=z_2$ and $z_{2k+1}=z_{2k}=z_2^k.$
\end{proof}

\medskip

In what follows we present $z_2$ in another form:
\begin{theorem}\label{z2}
Let $H$ be a semisimple Hopf algebra over an algebraically closed field $k$ of characteristic $0.$  Then  $$z_2=\sum \gL^1_1\gL_1^2S\gL^1_2S\gL_2^2=\sum_i \frac{1}{d_i^2} E_i\in Z(H).$$
\end{theorem}
\begin{proof}
By Theorem \ref{central22}, $z_2\in Z(H).$  Hence,
\begin{eqnarray*}\lefteqn{z_2=\sum (\gL^1\ad \gL_1^2)S\gL_2^2=}\\
&=&\sum_i\frac{1}{d_i}\left\langle\chi_i,(\gL^1\ad \gL_1^2)S\gL_2^2\right\rangle E_i\\
&=& \sum_i\frac{1}{d_i}\left\langle\chi_i\leftharpoonup(\gL^1\ad \gL_1^2),S\gL_2^2\right\rangle E_i\\
&=& \sum_i\frac{1}{d_i^2}\left\langle\chi_i,\gL^1\ad \gL_1^2\right\rangle\left\langle\chi_i,S\gL_2^2
\right\rangle E_i\qquad\text{(by \eqref{act})}\\
&=& \sum_i\frac{1}{d_i^2}\left\langle\chi_i,\gL_1^2\right\rangle\left\langle\chi_i,S\gL_2^2
\right\rangle E_i\qquad\text{(since $\chi_i$ is cocommutative)}\\
&=&\sum_i \frac{1}{d_i^2}\left\langle\chi_is(\chi_i),\gL^2\right\rangle E_i\\
&=&\sum_i\frac{1}{d_i^2}E_i\qquad\text{(by orthogonality of characters, \eqref{lacha})}
\end{eqnarray*}
\end{proof}
As a consequence of Theorem \ref{central22} and Theorem \ref{z2} we obtain:
\begin{coro}\label{z2n}
Let $H$ be a semisimple Hopf algebra.  Then
$$z_{n}=\sum_i \frac{1}{d_i^{n-n({\rm mod}2)}}E_i.$$
\end{coro}

An important consequence is the following theorem:
\begin{theorem}\label{z2com}
Let $H$ be a semisimple Hopf algebra over an algebraically closed field $k$ of characteristic $0.$ Then $H$ is commutative if and only if $z_n\in k$ for some $n\ge 2,$ if and only if $z_n\in k$ for all $n\ge 2.$
\end{theorem}
\begin{proof}
By Theorem \ref{central22} it is enough to consider $z_{2n}.$ If $z_{2n}=\ga\in k,$ then $z_{2n}=\ga\sum E_i,$ hence by  Corollary \ref{z2n},  $\ga=\frac{1}{d_i^{2n}}$ for all $i,$ implying all $d_i$ are equal. Since $V_0=k$ we have $d_0=1$ and so $d_i=1$ for all $i.$ This implies that $H$ is commutative. The converse is trivial.
\end{proof}

\medskip
For a group $G$ we have $z_2=\sum_{g,h\in G}ghg\minus h\minus$ hence it is easy to see that any commutator belongs to $z_2\leftharpoonup kG^*.$ This gives rise to the following question:
\begin{que}\label{questioncon}
Let $H$ be a semisimple Hopf algebra over an algebraically closed field $k$ of characteristic $0.$ Is that true that $${\rm Com}=z_2\leftharpoonup H^*?$$
\end{que}
While the answer to this question is unknown to us, we can show the following:

\begin{theorem}\label{z2com2}
Let $H$ be a semisimple Hopf algebra over an algebraically closed field $k$ of characteristic $0$ and let $n\ge 2.$ Then the commutator sublagebra of $H$ is the algebra generated by the left coideal
$$H'_n=z_n\leftharpoonup H^*.$$
\end{theorem}
\begin{proof}
 Since $z_2\in H'$ by Proposition \ref{h'}, and since $z_n=z_2^{\frac{n-n({\rm mod}2)}{2}}$ by Theorem \ref{central22}, it follows that $z_n\in H'$ and so  $H'_n\subset H'.$ Let $N$ be the algebra generated by $H'_n.$ Since $H'_n$ is ad-stable by \cite[Prop.2.5]{cw4}, it follows that $N$ is a normal left coideal subalgebra of $H,$ and $N\subset H'.$

Let $\pi:H\longrightarrow \ol{H}$ be the natural Hopf projection where $\ol{H}=H/HN^+.$ Then since $\ol{\gL}=\pi(\gL)$ and $z_n=\sum \gL^1_1\cdots\gL^n_1S\gL^1_2\cdots S\gL_2^n\in N,$ it follows that  $\sum\ol{ \gL^1_1\cdots\gL^n_1S\gL^1_2\cdots S\gL_2^n}=
  \ol{1}.$ This implies by Theorem \ref{z2com} that $\ol{H}$ is commutative, hence $H'\subset N$ and we are done.
\end{proof}
\medskip
The following summarizes properties of $z_2$ which are a direct consequence of this section.
\begin{coro}\label{sumz2}
Let $H$ be a semisimple Hopf algebra and let $z_2=\sum\gL^1_1\gL^2_1S\gL_2^1S\gL^2_2.$ Then:

\medskip\noin{\rm 1}. $z_2=\sum_i\frac{1}{d_i^2}E_i$ is an invertible central element of $H$ and $Sz_2=z_2.$

\medskip\noin{\rm 2}. $z_2^k=z_2^l,\,k\ne l,\,k,l\in \Z$ if and only if $H$ is commutative.

\medskip\noin{\rm 3}. $\langle \chi_i,z_2\rangle =\frac{1}{d_i}$ for all $0\le i\le n-1.$

\medskip\noin{\rm 4}. $\{\gL,z_2^k\}\in Z(H)$ for all $k\in\Z.$
\end{coro}

\medskip
As for groups, we can define the iterated commutator $\{\{H,H\},H\}$ for Hopf algebras $H.$ In this case we can extend Proposition \ref{h'} as follows:
\begin{proposition}\label{iterated}
Let $H$ be a semisimple Hopf algebra. Then the iterated commutator satisfies:
$$\{\{H,H\},H\}=k\Leftrightarrow \{H,H\}\subset Z(H).$$
\end{proposition}
\begin{proof}
Assume  $\{H,H\}\subset Z(H).$ Since $\{H,H\}$ is a left coideal and $Z(H)$ is $S$-stable, we have for all $a,b,c\in H,$
$$\sum\{a,b\}_1c_1S\{a,b\}_2Sc_2=\sum\{a,b\}_1S\{a,b\}_2c_1Sc_2=\langle\gep,abc\rangle.$$
Assume now $\{\{H,H\},H\}=k.$ Then for all $x\in \{H,H\},\,h\in H$ we have:
$$SxSh=\sum Sh_1Sx_1\underbrace{x_2h_2Sx_3Sh_3}_{\in \{\{H,H\},H\}}=ShSx.$$
Hence $x\in Z(H).$
\end{proof}

\section{Analogues of counting functions for Hopf algebras}
In what follows we indicate how counting functions can be realized from our point of view. Let $G$ be a finite group. Since $k$ is of characteristic $0$, $kG$ is a semisimple Hopf algebra with an idempotent integral $\gL=\frac{1}{|G|}\sum_{g\in G}g.$ The  projections $\{p_g\in kG^*\}$ defined by $\langle p_g,h\rangle=\gd_{g,h}$ form a basis for $kG^*$ dual to the natural basis  $\{g|g\in G\}$ of $kG.$ Since the counting function $N_w$ is a class function it follows that it is an element of $R(kG).$

Observe that
$$N_w=\sum_{(g_1\dots,g_r)\in G^r}p_{w(g_1\dots,g_r)}.$$

The following formula appear in the literature (see e.g.\cite[(2)]{av}.) We show it here using a Hopf algebraic approach.

\begin{eqnarray*}
\lefteqn{N_w=}\\
&=& \frac{1}{d_i}\sum_i \langle N_w,E_i\rangle \chi_i\qquad\qquad\text{(by \eqref{z})}\\
&=& \sum_i \langle N_w,\gL\leftharpoonup s(\chi_i)\rangle\chi_i\qquad\text{(by \eqref{dual1})}\\
&=& \sum_i \langle s(\chi_i),\gL\leftharpoonup N_w\rangle\chi_i\\
&=&\sum_i \left\langle s(\chi_i),\frac{1}{|G|}\sum_{g\in G}g\;\,\leftharpoonup \sum_{(g_1\dots,g_r)\in G^r} p_{w(g_1\dots,g_r)}\right\rangle\chi_i\\
&=&\sum_i \left\langle s(\chi_i),\frac{1}{|G|}\sum_{(g_1\dots,g_r)\in G^r} w(g_1\dots,g_r)\right\rangle\chi_i\\
\end{eqnarray*}
This formula yields known formulations of some counting functions:
\begin{example}\label{examples}
\medskip\noin{\rm 1}. {\bf Root functions}:
 if $w=x^m$ then $N_w$ is the so called $m$-th root function counting the number of solutions in $G$ to the equation $x^m=g,\,g\in G.$ In this case
$$N_w=\sum_{g\in G} p_{g^m}=\frac{1}{|G|}\sum_i\left\langle s(\chi_i),\sum_{g\in G}g^m\right\rangle\chi_i=\frac{1}{|G|}\sum_i\left\langle \chi_i,\sum_{g\in G}g^m\right\rangle\chi_i.$$
The coefficient of $\chi_i$ is called the $m$-th Frobenius-Schur indicator.

\medskip\noin{\rm 2}. {\bf Frobenius function for commutators}: If $w=[x,y]=xyx\minus y\minus$ then
$$N_w=\sum_i \left\langle\chi_i,\frac{1}{|G|}\sum_{(g,h)\in G^2} ghg\minus h\minus \right\rangle\chi_i.$$
But in terms of \eqref{zn}, $\sum_{(g,h)\in G^2} ghg\minus h\minus=|G|^2z_2,$ hence by Theorem \ref{z2} and Corollary \ref{sumz2}.3,
$$N_w=\sum_i \left\langle\chi_i,|G|z_2 \right\rangle\chi_i= \sum_i \frac{|G|}{d_i}\chi_i.$$

\medskip\noin{\rm 3}. {\bf Generalized commutators} (see \cite{ta}): Similarly to the previous example, the counting function corresponding to the generalized commutator  $w_n=[x_1,\dots,x_n]=x_1\cdots x_nx_1\minus\cdots x_n\minus$ is given by:
$$N_{w_n}=\sum_i \frac{|G|^{n-1}}{d_i^{n-1-n({\rm mod}2)}}\chi_i.$$
Also $N_{w_{n}}$ is the counting function corresponding to products of $n$ commutators.

\medskip\noin{\rm 4}. {\bf Iterated commutators} (see \cite{av}): The counting function for the iterated commutator $w=[[x,y],z]$ can be calculated as follows:
$$N_w=\sum_i \left\langle s(\chi_i),\frac{1}{|G|}\sum_{(g_1,g_2,g_3)\in G^3}[[g_1,g_2],g_3]\right\rangle\chi_i=
|G|^2\sum_{i,j} \frac{1}{d_id_j}\left\langle \chi_i s(\chi_i)\chi_j,\frac{1}{|G|}\sum_{g\in G}g\right\rangle\chi_i.$$
\end{example}

\medskip

In Hopf algebras we do not know if we are counting anything. However, these functions have a meaning as elements in $R(H)$ related to commutators, bilinear forms, Casimir elements and higher Frobenius Schur indicators.

\medskip\noin{\bf General root functions}:
A useful tool in studying the representations of Hopf algebras has been the introduction of Frobenius-Schur indicators, extending classical results of finite groups (see e.g \cite{ksz}). For a Hopf algebra $H$  define the $m$-th Sweedler power of $h\in H$ as $h^{[m]}=\sum h_1h_2\dots h_m,\,m\ge 1.$ For  an irreducible character $\chi,$ the $m$-th Frobenius-Schur indicator is given by $\nu_m(\chi)=\left\langle\chi,\gL^{[m]}\right\rangle.$

If $H=kG$ then $\gL^{[m]}=\frac{1}{|G|}\sum_{g\in G}g^m$ and so the function in Example \ref{examples}.1 becomes
$$r_m = \sum_i \nu_m(\chi_i)\chi_i=\Psi(\gL^{[m]}).$$

\medskip \noin{\bf Commutators and generalized commutators}. Set
\begin{equation}\label{frob}f_{rob}=\sum_i \frac{d}{d_i}\chi_i.\end{equation}

We show:
\begin{proposition}\label{frobn}
Let $H$ be a semisimple Hopf algebra,  let $z_2$ be the commutator as defined in \eqref{zn} and let $f_{rob}$ be defined as in \eqref{frob}. Then $$d\Psi(z_2)=f_{rob}.$$
\end{proposition}
\begin{proof}
Since $\Psi(E_i)=d_i s(\chi_i)$ by \eqref{dual1} and  $z_2=\sum_i\frac{1}{d_i^2}E_i$ by Theorem \ref{z2},  it follows that
$$d\Psi(z_2)=f_{rob}.$$
\end{proof}

For groups, Frobenius counting function for product of $n$ commutators can be extended to a formula counting $2n$-touples $(a_1,b_1,\dots,a_n,b_n)$  such that $g=a_1b_1a_1\minus b_1\minus\cdots a_nb_na_n\minus b_n\minus$ (\cite{km}). This number is given explicitly by $f_n(g)$ where:
\begin{equation}\label{fng}f_n=\sum (\frac{d}{d_i})^{2n-1}\chi_i.\end{equation}

The formula above is related also to the convolution product defined on functions on $G.$
Reminiscent of this product,  a convolution product $\bullet$ inside $H^*$ is defined as follows:
$$p\bullet p'=d\Psi(\Psi\minus(p)\Psi\minus(q)).$$

Variations of this product appear in \cite{an,cw1,pw}. This product should not be confused with the so called convolution product $*$ introduced by Kostant \cite{kos} which is the dual of the coproduct in $H$ and is considered to be the usual product in $H^*.$

It turns out that for Hopf algebras the convolution product of $f_{rob}$ is related to products of commutators as follows:
\begin{theorem}\label{bulletl}
Let $H$ be a semisimple Hopf algebra over an algebraically closed field $k$ of characteristic $0.$  Then for any $l>0,$
$$d^{2l-1}\Psi(z_{2l})=\sum_i\frac{d^{2l-1}}{d_i^{2l-1}}\chi_i=\underbrace{f_{rob}\bullet f_{rob}\cdots\bullet f_{rob}}_{l}.$$
\end{theorem}
\begin{proof}
For $l=1$ this is shown in Proposition \ref{frobn}. In other words,
$$\Psi\minus(f_{rob})=dz_2.$$
Since $z_2^2=\sum\frac{1}{d_i^4}E_i$ by Theorem \ref{z2}, it follows that $\Psi(z_2^2)=\sum\frac{1}{d_i^3}\chi_i.$ Hence
$$f_{rob}\bullet f_{rob}=d\Psi(\Psi\minus(f_{rob})\Psi\minus(f_{rob}))=d^3\Psi(z_2^2)=\sum_i \frac{d^3}{d_i^3}\chi_i.$$
This proves the case of $l=2.$ The result follows now by induction.
\end{proof}
\medskip

Another point of view of the setup above is to realize the special commutators $z_n$ as Casimir elements of certain symmetric bilinear on $Z(H)$ or $H.$ The symmetric bilinear form on $H$ defined in  \eqref{bilinear}
induces a non-degenerate symmetric bilinear form on $Z(H):$
$$\gb^Z(z,z')=\left\langle\gl,zz'\right\rangle$$

Let $\{E_i\}$ be the set of central idempotents of $H.$ Since $\left\langle\gl,E_iE_j\right\rangle=\gd_{ij}d_i^2,$ it follows that
the  Casimir element and the central Casimir element corresponding to $\gb^Z$ are given by:
\begin{equation}\label{cas0} \sum \frac{1}{d_i^2} E_i\ot E_i\quad\text{and}\quad Cas^Z_{\gb}=\sum_i\frac{1}{d_i^2}E_i=z_2.\end{equation}

By using \eqref{dual1} it is straightforward to see that:
\begin{lemma}\label{symmetric}
Let $H$ be a semisimple Hopf algebra  and let $t=\sum \ga_i\chi_i,\,\ga_i\in k.$ Then:

\medskip\noin{\rm 1}. $t\leftharpoonup Z(H)=R(H)$ if and only if $\ga_i\ne 0$ for all $i.$

\medskip\noin{\rm 2}. If indeed $\ga_i\ne 0$ for all $i,$ then the central Casimir element $Cas^Z_t$  of the corresponding bilinear form $\gb_t$
is:
$$Cas^Z_t =\sum \frac{1}{\ga_i d_i}E_i.$$ If moreover, $\ga_i=\ga_{i^*}$ (where  $\ga_{i^*}$ is the coefficient of $s(\chi_i)$), then:
$$\Psi(Cas^Z_t)=\sum_i \frac{1}{\ga_i}\chi_i.$$
\end{lemma}

Summarizing results from Lemma \ref{symmetric}, Theorem \ref{bulletl} and Corollary \ref{z2n} we obtain the following realization of the counting functions in Example \ref{examples}.3:
\begin{coro}\label{casn1}
Let $H$ be a semisimple Hopf algebra and let $z_n,\,n\ge 2$ be defined as in \eqref{zn}.  Define the symmteric form $\gb_n^Z$ on $Z(H)$  through:
 $$t_n=\sum_id_i^{n-1-n({\rm mod}2)}\chi_i.$$
Then the corresponding central Casimir element satisfies:
$$Cas^Z_{\gb_n}=z_n.$$
Moreover,
$$d^{n-1}\Psi(Cas^Z_{\gb_n})=d^{n({\rm mod}2)}\underbrace{f_{rob}\bullet\cdots\bullet f_{rob}}_{\frac{n-n({\rm mod}2)}{2}}=\sum_i\frac{d^{n-1}}{d_i^{n-n({\rm mod}2)}}\chi_i.$$

Note that $t_2=\gl$ and $Cas^Z_{\gb_2}=z_2.$
\end{coro}

The elements $z_n$ are related also to symmetric forms on $H$ as follows:

\begin{proposition}\label{second}
Let $H$ be a semisimple Hopf algebra and let $z_n,\,n\ge 2$ be defined as in \eqref{zn}. Then $z_n$  is the central Casimir element of the symmetric form on $H$ $\tl{\gb}_n$ defind through:
$$\tl{t}_n=\sum_i d_i^{n+1-n({\rm mod}2)}\chi_i,\qquad \tilde{\gb}_n^H(h,h')=\langle \tl{t}_n ,hh'\rangle$$ for all $h,h'\in H.$

The Higman map defined by this form (as in \eqref{tau}) is
$$\tau_{\gb_n}(h)=Z_{n+1-n({\rm mod}2)}(h).$$
\end{proposition}
\begin{proof}
Take $u_n=\sum_i d_i^{n-n({\rm mod}2)}E_i=z_n\minus.$ Then $\tl{t}_n=\gl\leftharpoonup u_n.$ Since $u_n\minus=\sum\gL_1^1\cdots\gL_1^nS\gL_1^2\cdots S\gL^n_2,$ it follows from \eqref{connection} and \eqref{casimir} that the corresponding central Casimir element is $\gL\ad z_n=z_n$  since $z_n$ is central.

The Higman map  is given now by:
$$\tau_{\tl{\gb}_n}(h)=\sum\gL_1z_nhS\gL_2=z_n(\gL\ad h),$$
where the last equality follows since $z_n\in Z(H).$ By Theorem \ref{central22}
$$z_n(\gL\ad h)=z_2^{\frac{n-n({\rm mod}2)}{2}}(\gL\cdot h)=Z_{n+1-n({\rm mod}2)}(h).$$
\end{proof}

\medskip
\noin{\bf Iterated commutators}. Motivated by the counting function for iterated commutators for groups, (see \cite{av}), set:
\begin{equation}\label{nxyz}f_{\{\{x,y\},z\}}=d^2\sum_{i,j} \frac{1}{d_id_j}\langle \chi_i s(\chi_i)\chi_j,\gL\rangle\chi_i.\end{equation}
Then we have:
\begin{theorem}\label{z2lamda}
Let $H$ be a semisimple hopf algebra over an algebraically closed fiels of characteristic $0.$  Then:
$$f_{\{\{x,y\},z\}}=d^2\sum_{i,j} \frac{1}{d_i}\langle \chi_i s(\chi_i),z_2\rangle\chi_i=d^2\Psi(\{z_2,\gL\}).$$
\end{theorem}
\begin{proof}
We have:
\begin{eqnarray*}
\lefteqn{\langle \chi_i,\{z_2,\gL\}\rangle=}\\
&=&\sum_{z_2,\gL} \langle \chi_i,z_{2_1}\gL_1Sz_{2_2}S\gL_2\rangle\\
&=&\sum_{z_2,\gL} \langle \frac{1}{d_i}\chi_i,Sz_{2_2}\rangle\langle\chi_i,z_{2_1}\rangle\qquad\text{(by \eqref{act} since $\gL\cdot Sz_{2_2}\in Z(H)$)}\\
&=&\frac{1}{d_i}\langle \chi_is(\chi_i),z_2\rangle.
\end{eqnarray*}
Since $\{z_2,\gL\}\in Z(H)$ by Corollary \ref{sumz2}.4, we have by \eqref{act},
$$\Psi(\{z_2,\gL\})=\sum_i d_i\chi_i\leftharpoonup \{z_2,\gL\}=\sum_i \langle\chi_i,\{z_2,\gL\}\rangle\chi_i=\sum_i\frac{1}{d_i}\langle \chi_is(\chi_i),z_2\rangle\chi_i.$$
This proves the second equality in the theorem.
As for the first equality,  since
$$z_2=\sum_j \frac{1}{d_j^2}E_j=\gL\leftharpoonup\sum_j\frac{1}{d_j}\chi_j$$ it follows that:
$$\Psi(\{z_2,\gL\})= d^2\sum_{i,j} \frac{1}{d_i}\left\langle \chi_i s(\chi_i),\gL\leftharpoonup\frac{1}{d_j}\chi_j\right\rangle\chi_i
=d^2\sum_{i,j} \frac{1}{d_id_j}\langle \chi_i s(\chi_i)\chi_j,\gL\rangle\chi_i.$$
By \eqref{nxyz} we are done.
\end{proof}

\section{Commutators for almost cocommutative Hopf algebras}

It is known (\cite{k,z}) that when $H$ is semisimple then so is $R(H).$  Let $\{\frac{1}{d}\gl=F_0,\dots F_{m-1}\}$ be a complete set of central primitive idempotents of $R(H),$ and let $\{f_0,\dots,f_{m-1}\}$ be primitive orthogonal idempotents in $R(H)$ so that $f_iF_j=\gd_{ij}f_i.$ Define the {\bf conjugacy class} ${\mathfrak C}_i$ as:
\begin{equation}\label{cci}{\mathfrak C}_i=\gL\leftharpoonup f_iH^*.\end{equation}
Then we have shown in \cite{cw4} that:

  {\it ${\mathfrak C}_i$ is an irreducible $D(H)$-module and moreover,
$$H\cong \oplus_{i=0}^{m-1} {\mathfrak C}_i^{\oplus m_i}$$
as $D(H)$-modules}.

We generalize also the notions of  {\bf Class sum}  and of a representative of a conjugacy class as follows:
\begin{equation}\label{ci} C_i=\gL\leftharpoonup dF_i\qquad \eta_i=\frac{C_i}{\dim(f_iH^*)}.\end{equation}
We refer to $\eta_i$ as a  {\bf normalized class sum}.

If $R(H)$ is commutative (which is equivalent to $H$ being almost cocommutative \cite{ni}), then $\{F_i\}$ forms a basis of $R(H)$ and $f_i=F_i$ for all $i.$ In this case
$\{F_i\},\{\eta_i\}$ is  another pair of dual bases for $R(H)$ and $Z(H)$ respectively.

\medskip
In this context $z_2$ has another realization.
\begin{proposition}\label{cas5}
Let $H$ be a semisimple Hopf algebra and assume $R(H)$ is commutative. Then
$$z_2=\frac{1}{d}\sum_i \dim (F_iH^*)\eta_iS\eta_i$$
\end{proposition}
\begin{proof}
Since $\eta_i=\Psi\minus\left(\frac{d}{\dim (F_iH^*)}F_i\right),$ it follows that $\Psi(\eta_i)= \frac{d}{\dim (F_iH^*)}F_i.$ A direct computation shows  that $\left\langle F_i,\gL\right\rangle =\frac{\dim(F_iH^*)}{d}$ (see e.g. \cite{cw3}). Hence
$$\left\langle\gl,\eta_iS\eta_j\right\rangle=\langle\Psi(\eta_i),\eta_j\rangle=\left\langle \frac{d}{\dim (F_iH^*)}F_i,\gL\leftharpoonup\frac{d}{\dim (F_jH^*)}F_j\right\rangle=\gd_{ij}\frac{d}{\dim (F_iH^*)}.$$

Let $\gb_2$ be the bilinear form associated with $t_2=\gl.$  Then $\{\eta_i\},\{\frac{dim (F_iH^*)}{d} S\eta_i\}$ form a dual basis for  $\gb_2.$ The result follows now from Corollary \ref{casn1}.
\end{proof}

Observe that for a group $G,$ we have $g$ is a commutator if and only if $f_{rob}(g)\ne 0.$ Moreover, $g\in G'$ if and only if $f_n(g)\ne 0$ for some $n\ge 1,$ where $f_n$ is as defined in \eqref{fng}. Similarly, $g$ is an iterated commutator if and only if $f_{\{\{x,y\}z\}}(g)\ne 0,$ where $f_{\{\{x,y\}z\}}$ is as defined in \eqref{nxyz}.  Hence also for a class sum $C,\;f_{rob}(C)\ne 0$ (respectively $f_n(C)\ne 0,\, f_{\{\{x,y\}z\}}(C)\ne 0$) if and only if $C$ is a sum of commutators ( respectively $C\in kG',\;C$ is a sum of iterated commutators).

\medskip Since any ad-stable left coideal $L$ is a $D(H)$-module and since the ${\mathfrak C}_i$ are all the irreducible $D(H)$-submodules of $H,$ it follows that  $L$ is a sum of some of the ${\mathfrak C}_i$'s. As such $L$ is generated as a coideal by some of the $\eta_i$'s.  Since $H'$ is an ad-stable left coideal of $H,$ the following question is relevant:

\begin{que}\label{question}
Let $H$ be a semisimple Hopf algebra over an algebraically closed field of characteristic $0,$ and let $\eta$ be a normalized class sum. Is it true that $f_n(\eta)\ne 0$ if and only if $\eta\in H'?$
\end{que}
Recall $H'$ is generated as an algebra by $\{z_n\leftharpoonup H^*\}$ (Theorem \ref{z2com2}). Also recall that both $H'$ and  $\{z_n\leftharpoonup H^*\}$ are $D(H)$-modules hence are direct sums of conjugacy classes. As such they are obtained from certain normalized class sums.  So a step in answering the question above is based on the following lemma:
\begin{lemma}\label{P} Assume $H$ is almost cocommutative and let $f=\sum_i\ga_iF_i\in R(H)$ and $P_f=\Psi\minus(f)\leftharpoonup H^*.$ Then:
$$\eta_j\in P_f\Longleftrightarrow \langle f,\eta_j\rangle\ne 0.$$
\end{lemma}
\begin{proof}
Assume $\eta_j\in P_f.$ Then $\eta_j=\Psi\minus(f)\leftharpoonup q=\Psi\minus(fq).$ By the definition of $\eta_j$ it follows that
$$fq=\frac{d}{\dim (F_iH^*)}F_j.$$
 Since both $\eta_j,\,\Psi\minus(f)\in Z(H)$ it follows that $q\in R(H).$ Hence  we have
$q=\sum\gb_iF_i,\,\gb_i\in k$ and so  $fq=\sum\ga_i\gb_iF_i.$ This implies that $$\ga_j\gb_j= \frac{d}{\dim (F_iH^*)},$$ in particular
$\langle f,\eta_j\rangle=\ga_j\ne 0.$

Conversely, assume $\langle f,\eta_j\rangle\ne 0.$ Then $\ga_jF_j=fF_j\ne 0,$ which implies that  $\eta_j\in P_f.$
\end{proof}

As a consequence we obtain:
\begin{theorem}\label{eta}
Let $H$ be an almost cocommutative semisimple Hopf algebra over an algebraically closed field of characteristic $0,$ let $\eta_j$ be a normalized class and let $f_n$ be as in \eqref{fng}. Then
$$\eta_j\in z_2^n\leftharpoonup H^*\Longleftrightarrow \langle f_n,\eta_j\rangle\ne 0.$$
\end{theorem}
\begin{proof}
The result follows from Lemma \ref{P} with $f=f_n.$
\end{proof}

Since $\Psi\minus(f_{\{\{x,y\},z\}})=d^2\{z_2,\gL\}$ by Theorem \ref{z2lamda}, another consequence of Lemma \ref{P} is:
\begin{coro}\label{iterated1}
If $H$ is semisimple and almost cocommutative then:
$$\eta_j\in \{z_2,\gL\}\leftharpoonup H^*\Longleftrightarrow \langle f_{\{\{x,y\},z\}},\eta_j\rangle\ne 0,$$
\end{coro}

\medskip

\begin{rema}\label{iterated2}
 When $R(H)$ is commutative we can  use the character table to write $\chi_is(\chi_i)$ as a sum of the irreducible characters, and then compute $\langle\chi_is(\chi_i),z_2\rangle$ by using Corollary \ref{sumz2}.3.
\end{rema}
\medskip
The known conjecture of Kaplansky is that if $H$ is semisimple then $d_i|d.$ It is known to be true for various families of Hopf algebras, in particular when $H$ is a semisimple quasitriangular Hopf algebra \cite{eg}.  In this case we have:
\begin{theorem}\label{kap}
Let $H$ be a semisimple Hopf algebra over an algebraically closed field of characteristic $0,$ and assume $d_i|d$ for any irreducible $H$-module $V_i.$ Then the  functions $f_{rob},\,f_n$ and $f_{\{\{x,y\},z\}}$ given in \eqref{frob}, \eqref{fng} and \eqref{nxyz} are characters.
\end{theorem}
\begin{proof}
The functions $f_{rob}$ and $f_n$ are clearly characters. We show it for  $f_{\{\{x,y\},z\}}$ as well. Let $\chi_is(\chi_i)=\sum m_{ki}\chi_k,\;m_{ki}\in\Z^+.$ By Corollary \ref{sumz2}.3, $\langle\chi_is(\chi_i),z_2\rangle=\sum_k\frac{m_{ki}}{d_k}.$ Hence by Theorem \ref{z2lamda},
$$f_{\{\{x,y\},z\}}=d^2\sum_i \frac{1}{d_i}\langle \chi_is(\chi_i),z_2\rangle\chi_i=\sum_{i,k} \frac{d^2}{d_id_k}m_{ki}\chi_i.$$
Since all the coefficients are integers we are done.
\end{proof}
\begin{rema}\label{susan}
The root function for Hopf algebras is not necessarily a character for Hopf algebras satisfying the hypothesis of Theorem \ref{kap}. In \cite{imm} the authors found groups $G$ so that the Frobenius-Schur indicators of $D(G)$ are not  integers and so the appropriate root function is not a character.
\end{rema}

\medskip
If $(H,R)$ is also factorizable, more can be said. Recall \cite{cw4} that when $H$ is factorizable we have
\begin{equation}\label{fqchi}f_Q(\chi_i)=d_i\eta_i\end{equation} for all $i>0,$
where $f_Q(p)=\sum \langle p, R^2R^1\rangle R^1R^2$ is the Drinfeld map.

We can show:
\begin{proposition}\label{fact}
Let $H$ be a factorizable semisimple Hopf algebra. Then:

\medskip\noin{\rm 1}. Let $\chi_{ad}=\sum_i \chi_i s(\chi_i)$ be the character of the left adjoint representation of $H.$ Then $$f_Q(\chi_{ad})=dz_2.$$

\medskip\noin{\rm 2}. Let $F=f_Q\Psi$ and $F'=\Psi f_Q$ be the quantum Fourier Transforms defined on $H$ and $H^*$ respectively (see\cite{lm,cw1}). Then $$F(z_2)=\sum_i\eta_i\qquad F'(\chi_{ad})=f_{rob}.$$
\end{proposition}
\begin{proof}
1. Follows from Proposition \ref{cas5}, \eqref{fqchi},  the fact that $f_Q$ is multiplicative on $R(H)$ and $\dim(F_iH^*)=\dim(E_iH^*)=d_i^2.$

2. By Theorem \ref{z2}, $z_2=\sum_i\frac{1}{d_i^2}E_i,$ hence we have:
$$f_Q\Psi (z_2)=f_Q\left(\sum_i \frac{1}{d_i}\chi_i\right)=\sum_i\eta_i.$$
The second part follows directly from part 1.
\end{proof}

\end{document}